\documentclass{amsart}
\usepackage{amsmath, amssymb}
\usepackage[all]{xy}
\usepackage[colorlinks=true]{hyperref}

\theoremstyle{plain}
 \newtheorem{theorem}{Theorem}[section]
 \newtheorem{lemma}{Lemma}[section]

 \newtheorem{conjecture}{Conjecture}[section]

\theoremstyle{definition}

\newcommand{\mbbZ}{\mathbb{Z}}
\newcommand{\mbbR}{\mathbb{R}}

\title{On some conjectures of Samuels and Feige}

\author{Roland Paulin}
\address{Max-Planck-Institut f\"ur Mathematik \\
Vivatsgasse 7 \\
D-53111, Bonn \\
Germany}

\email{paulinroland@gmail.com}
\thanks{}

\subjclass[2010]{}
\date{\today}

\begin{document}

\begin{abstract}
Let $\mu_1 \ge \dotsc \ge \mu_n > 0$ and $\mu_1 + \dotsm + \mu_n = 1$.
Let $X_1, \dotsc, X_n$ be independent non-negative random variables with $EX_1 = \dotsc = EX_n = 1$, and let $Z = \sum_{i=1}^n \mu_i X_i$.
Let $M = \max_{1 \le i \le n} \mu_i = \mu_1$, and let $\delta > 0$ and $T = 1 + \delta$.
Both Samuels and Feige formulated conjectures bounding the probability $P(Z < T)$ from above.
We prove that Samuels' conjecture implies a conjecture of Feige.
\end{abstract}

\maketitle

\section{Introduction}

Let $n \in \mbbZ_{\ge 1}$, $\mu_1 \ge \dotsc \ge \mu_n > 0$ and $\mu_1 + \dotsm + \mu_n = 1$.
Let $X_1, \dotsc, X_n$ be independent non-negative random variables with $EX_1 = \dotsc = EX_n = 1$, and let $Z = \sum_{i=1}^n \mu_i X_i$.
Let $M = \max_{1 \le i \le n} \mu_i = \mu_1$, and let $\delta > 0$ and $T = 1 + \delta$.

Feige \cite{Feige}, \cite{Feige-preprint} proves a weaker version of the following conjecture, with the constant $\frac{1}{13}$ in place of $\frac{1}{e}$.
\begin{conjecture} \label{Conj:weak}
\[
P(Z < T) \ge \min\left(\frac{\delta}{\delta + M}, \frac{1}{e}\right).
\]
\end{conjecture}
We can have equality here sometimes.
If $M = \max_{1 \le i \le n} \mu_i = \mu_1$, and $X_2 = \dotsc = X_n = 1$, and $P(X_1 = \frac{\delta + M}{M}) = \frac{M}{\delta + M}$ and $P(X_1 = 0) = \frac{\delta}{\delta + M}$, then $P(Z < T) = \frac{\delta}{\delta + M}$.
If $M = \mu_1 = \dotsc = \mu_n = \frac{1}{n}$, and $P(X_i = n(1+\delta)) = \frac{1}{n(1+\delta)}$ and $P(X_i = 0) = 1-\frac{1}{n(1+\delta)}$ for every $i \in \{1, \dotsc, n\}$, then $P(Z<T) = (1-\frac{1}{n(1+\delta)})^n$.
Taking $n \to \infty$ shows that the conjecture is not true for any constant bigger than $\frac{1}{e}$.

Samuels \cite{Samuels} has a related conjecture.
\begin{conjecture} \label{Conj:strong}
\[
P(Z < T) \ge \min_{1 \le i \le n} \prod_{j=1}^i \left(1 - \frac{\mu_j}{T - \sum_{k=i+1}^n \mu_k}\right).
\]
\end{conjecture}
Here we can have equality for all $\mu_1 \ge \dotsc \ge \mu_n > 0$: Let $i \in \{1, \dotsc, n\}$, let $X_j = 1$ for every $j > i$, and if $j \le i$, then let $P(X_j = 0) = 1 - \frac{\mu_j}{T - \sum_{k=i+1}^n \mu_k}$ and $P\left(X_j = \frac{T - \sum_{k=i+1}^n \mu_k}{\mu_j}\right) = \frac{\mu_j}{T - \sum_{k=i+1}^n \mu_k}$.
Then $P(Z < T) = \prod_{j=1}^i (1 - \frac{\mu_j}{T - \sum_{k=i+1}^n \mu_k})$.

The goal of this paper is to prove the following theorem.
\begin{theorem} \label{Thm:main}
Conjecture \ref{Conj:strong} implies Conjecture \ref{Conj:weak}.
\end{theorem}

\section{Proof of the theorem}

Let
\[
\Phi \colon [0,1] \times (0, \infty) \to \mbbR, \quad (\mu,\rho) \mapsto \begin{cases}
\frac{1}{\mu} \log\left(1-\frac{\mu}{1+\rho}\right) & \textrm{ if } \mu \in (0,1], \\
-\frac{1}{1+\rho} & \textrm{ if } \mu = 0.
\end{cases}
\]
During the proof we will use several lemmas related to the function $\Phi$, these are collected together in the last section.

We assume that Conjecture \ref{Conj:strong} holds.
Let $\sigma_i = \sum_{k=1}^i \mu_k$ for every $i \in \{1, \dotsc, n\}$.
Then
\[
P(Z < T) \ge \min_{1 \le i \le n} \prod_{j=1}^i \left( 1 - \frac{\mu_j}{\sigma_i + \delta} \right).
\]
for every $i \in \{1, \dotsc, n\}$.
Take an $i$ which minimizes the right hand side.
Note that $\Phi\left(\frac{\mu_j}{\sigma_i}, \frac{\delta}{\sigma_i} \right) = \frac{\sigma_i}{\mu_j} \log\left(1-\frac{\mu_j}{\sigma_i+\delta}\right)$ for every $j \in \{1, \dotsc, i\}$.
So
\[
\log P(Z < 1 + \delta) \ge \sum_{j=1}^i \log \left( 1 - \frac{\mu_j}{\sigma_i + \delta} \right) = \sum_{j=1}^i \frac{\mu_j}{\sigma_i} \Phi\left(\frac{\mu_j}{\sigma_i}, \frac{\delta}{\sigma_i} \right).
\]
Since $\sum_{j=1}^i \frac{\mu_j}{\sigma_i} = 1$, and $\Phi\left(\frac{\mu_j}{\sigma_i}, \frac{\delta}{\sigma_i} \right) \ge \Phi\left(\frac{M}{\sigma_i}, \frac{\delta}{\sigma_i} \right)$ by Lemma \ref{Lemma:Phi-decreasing-in-mu}, we obtain
\[
\log P(Z < 1 + \delta) \ge \Phi\left(\frac{M}{\sigma_i}, \frac{\delta}{\sigma_i} \right).
\]
We have $M = \mu_1 \le \sigma_i \le \sum_{j=1}^n \mu_j = 1$, so $\frac{M}{\sigma_i} \in [M, 1]$, hence
\begin{align*}
\log P(Z < 1+\delta) &\ge \min_{\sigma \in [M,1]} \Phi\left(\frac{M}{\sigma}, \frac{\delta}{\sigma}\right) = \min_{t \in [M, 1]} \Phi\left(t, \frac{\delta}{M} t\right) \\
&= \min\left(\Phi\left(1, \frac{\delta}{M}\right), \Phi(M,\delta)\right)
\end{align*}
by Lemma \ref{Lemma:Phi-t-alphat-concave}.
So
\[
P(Z < 1+\delta) \ge \min\left( \frac{\delta}{\delta+M}, \left(1-\frac{M}{1+\delta}\right)^{\frac{1}{M}} \right) \ge \min\left( \frac{\delta}{\delta+M}, \frac{1}{e} \right),
\]
where the second inequality follows from Lemma \ref{Lemma:Phi-mu-rho-bound}.
So Conjecture \ref{Conj:strong} indeed implies Conjecture \ref{Conj:weak}.

\section{A few lemmas related to the function $\Phi$}

The following simple inequality will be used in the proof of Lemma \ref{Lemma:Phi-mu-rho-bound}.
\begin{lemma} \label{Lemma:basic-ineq}
If $t \in [0,1]$, then $1-\frac{t}{1+\frac{t}{e-1}} \ge e^{-t}$.
\end{lemma}
\begin{proof}
We need to show $1-e^{-t} \ge \frac{t}{1+\frac{t}{e-1}}$, or equivalently $(1+\frac{t}{e-1})(1-e^{-t}) \ge t$.
Let $f \colon \mbbR \to \mbbR$, $t \mapsto (1+\frac{t}{e-1})(1-e^{-t}) - t$.
Then $f$ is a smooth function, $f(0) = 0$, $f'(t) = \frac{1}{e-1} e^{-t} ((e-2)(1-e^t) + t)$ and $f''(t) = \frac{1}{e-1} e^{-t} (3-e-t)$.
So $f$ is convex in the interval $[0,3-e]$, and concave in the interval $[3-e,1]$.
Moreover $f'(0) = 0$, so $f$ is monotone increasing in the interval $[0,3-e]$, hence $f(t) \ge f(0) = 0$ for $t \in [0, 3-e]$, so in particular $f(3-e) \ge 0$.
Since $f$ is concave in the interval $[3-e,1]$, and $f(3-e) \ge 0$ and $f(1) = 0$, we get that $f(t) \ge 0$ for every $t \in [3-e, 1]$.
So indeed $f(t) \ge 0$ for every $t \in [0,1]$.
\end{proof}

The following lemma states that $\Phi$ is strictly decreasing in its first variable.
\begin{lemma} \label{Lemma:Phi-decreasing-in-mu}
If $\rho \in (0, \infty)$, $\lambda, \mu \in [0,1]$ and $\lambda < \mu$, then $\Phi(\mu,\rho) < \Phi(\lambda,\rho)$.
\end{lemma}
\begin{proof}
Let
\[
\varphi \colon [0,1) \to \mbbR, \quad s \mapsto -\sum_{k \ge 0} \frac{s^k}{k+1} = \begin{cases}
\frac{\log(1-s)}{s} & \textrm{ if } s \in (0,1) \\
-1 & \textrm{ if } s = 0.
\end{cases}
\]
The series expansion clearly shows that $\varphi$ is strictly decreasing in $[0,1)$.
Thus
\[
\Phi(\mu, \rho) = \frac{1}{1+\rho} \varphi\left(\frac{\mu}{1+\rho}\right) < \frac{1}{1+\rho} \varphi\left(\frac{\lambda}{1+\rho}\right) = \Phi(\lambda, \rho).
\]
\end{proof}

The following lemma gives a lower bound for $\Phi$, which is used at end of the proof of Theorem \ref{Thm:main}.
\begin{lemma} \label{Lemma:Phi-mu-rho-bound}
If $\mu \in [0,1]$ and $\rho \in (0, \infty)$, then
\[
\Phi(\mu,\rho) \ge \log \left(\min\left(\frac{\rho}{\mu+\rho}, \frac{1}{e}\right)\right).
\]
\end{lemma}
\begin{proof}
If $\mu = 0$, then $\Phi(\mu,\rho) = -\frac{1}{1+\rho} > -1 = \log \frac{1}{e}$.
If $\mu = 1$, then $\Phi(\mu,\rho) = \log \frac{\rho}{1+\rho} = \log \frac{\rho}{\mu+\rho}$.
Now fix a $\mu \in (0,1)$.
Note that $\Phi$ is smooth in $(0,1) \times (0,\infty)$.
Let $\partial_1 \Phi$ and $\partial_2 \Phi$ denote the partial derivatives of $\Phi$ with respect to the first and second variable.
We have $(\partial_2 \Phi)(\mu,\rho) = \frac{1}{(1+\rho)(1+\rho-\mu)} > 0$.

First suppose that $\rho \in [\frac{\mu}{e-1}, \infty)$.
Then $\frac{\rho}{\mu+\rho} \ge \frac{1}{e}$, so we need to show that $\Phi(\mu, \rho) \ge -1$.
Since $\rho \mapsto \Phi(\mu,\rho)$ is monotone increasing in $[\frac{\mu}{e-1}, \infty)$, we have $\Phi(\mu,\rho) \ge \Phi(\mu,\frac{\mu}{e-1}) = \frac{1}{\mu} \log\left( 1 - \frac{\mu}{1+\frac{\mu}{e-1}} \right) \ge -1$ by Lemma \ref{Lemma:basic-ineq}.

Now suppose that $\rho \in (0,\frac{\mu}{e-1}]$.
Then $\frac{\rho}{\mu+\rho} \le \frac{1}{e}$, so we need to show that $\Phi(\mu, \rho) \ge \log \frac{\rho}{\mu+\rho}$.
Let $g \colon (0,\infty) \to \mbbR$, $\rho \mapsto \Phi(\mu, \rho) - \log \frac{\rho}{\mu+\rho}$.
Then $g'(\rho) = -\frac{(1-\mu)((1+\rho)\mu-\rho^2)}{\rho(1+\rho)(1+\rho-\mu)(\mu+\rho)}$.
We have $\rho \le \frac{\mu}{e-1} < \mu$, so $(1+\rho)\mu - \rho^2 > (1+\rho)\rho - \rho^2 = \rho > 0$, hence $g'(\rho) < 0$ for every $\rho \in (0,\frac{\mu}{e-1}]$.
So $g(\rho) \ge g(\frac{\mu}{e-1}) = \Phi(\mu, \frac{\mu}{e-1}) + 1 \ge 0$ by the previous paragraph.
\end{proof}

The following lemma states a property of the function $\Phi$ which is essential for the proof of Theorem \ref{Thm:main}.
\begin{lemma} \label{Lemma:Phi-t-alphat-concave}
If $\alpha \in (0, \infty)$, then $h_{\alpha} \colon [0,1] \to \mbbR$, $t \mapsto \Phi(t, \alpha t)$ is a concave function.
\end{lemma}
\begin{proof}
The continuity of $\Phi$ in $[0,1] \times (0,\infty)$, and the smoothness of $\Phi$ in $(0,1) \times (0,\infty)$ implies the continuity of $h_{\alpha}$ in $[0,1]$, and the smoothness of $h_{\alpha}$ in $(0,1)$.
So it is enough to prove that $h_{\alpha}''(t) \le 0$ for every $t \in (0,1)$.
Let $t \in (0,1)$, then $h_{\alpha}(t) = \frac{1}{t} \log\left( 1 - \frac{t}{1 + \alpha t} \right)$.
One can check that
\[
\eta(t) := t^3 h''_{\alpha}(t) = \frac{t(2 + (6\alpha-3)t + 4(\alpha^2-\alpha)t^2)}{(1 - t + \alpha t)^2 (1 + \alpha t)^2} + 2 \log\left( 1 - \frac{t}{1 + \alpha t} \right).
\]
It is enough to show that $\eta(t) < 0$ for every $t \in (0,1)$.
Taking the derivative, we obtain
\[
\eta'(t) = -\frac{t^2 (1 + 3 (1 - 2 \alpha + 2(\alpha-\alpha^2)t)^2)}{2(1 - t + \alpha t)^3 (1 + \alpha t)^3}.
\]
Clearly $\eta'(t) < 0$.
So $\eta$ is strictly decreasing in the interval $(0,1)$, hence $\eta(t) < \lim_{s \to 0} \eta(s) = 0$ for every $t \in (0,1)$.
\end{proof}

\end{document}